\documentclass[sn-mathphys]{sn-jnl}

\jyear{2021}%

\theoremstyle{thmstyleone}%
\newtheorem{theorem}{Theorem}
\newtheorem{proposition}[theorem]{Proposition}%
\newtheorem{lemma}[theorem]{Lemma}

\theoremstyle{thmstyletwo}%
\newtheorem{example}{Example}%

\theoremstyle{thmstylethree}%
\newtheorem{definition}{Definition}%

\begin{document}

\title[Logarithmic Schr\"{o}dinger equations on locally finite graphs]{Ground states for logarithmic Schr\"{o}dinger equations on locally finite graphs}


\author*[1]{\fnm{Xiaojun} \sur{Chang}}\email{changxj100@nenu.edu.cn}
\equalcont{These authors contributed equally to this work.}

\author[1]{\fnm{Ru} \sur{Wang}}\email{wangr076@nenu.edu.cn}
\equalcont{These authors contributed equally to this work.}

\author[2]{\fnm{Duokui} \sur{Yan}}\email{duokuiyan@buaa.edu.cn}
\equalcont{These authors contributed equally to this work.}

\affil[1]{\orgdiv{School of Mathematics and Statistics \& Center for Mathematics and Interdisciplinary Sciences}, \orgname{Northeast Normal University},  \city{Changchun}, \postcode{130024},  \country{China}}

\affil[2]{\orgdiv{School of Mathematical Sciences}, \orgname{Beihang University},  \city{Beijing}, \postcode{100191}, \country{China}}


\abstract{In this paper, we study the following logarithmic Schr\"{o}dinger equation
\[
-\Delta u+a(x)u=u\log u^2\ \ \ \ \mbox{in }V,
\]
where $\Delta$ is the graph Laplacian, $G=(V,E)$ is a connected locally finite graph, the potential $a: V\to \mathbb{R}$ is bounded from below and may change sign. We first establish two Sobolev compact embedding theorems in the case when different assumptions are imposed on $a(x)$. It leads to two kinds of associated energy functionals, one of which is not well-defined under the logarithmic nonlinearity, while the other is $C^1$.
The existence of ground state solutions are then obtained by using the Nehari manifold method and the mountain pass theorem respectively.}

\keywords{Ground state solutions, logarithmic Schr\"{o}dinger equations, locally finite graphs, variational methods.}


\pacs[MSC Classification]{35A15, 35R02, 35Q55, 39A12}

\maketitle

\section{Introduction and main results}\label{sec1}

In this paper, we are interested in the following logarithmic Schr\"{o}dinger equation
\begin{equation}\label{eq-eq1}
-\Delta u+a(x)u=u\log u^2 \quad \mbox{in }V
\end{equation}
on a connected locally finite graph $G=(E,V)$, where $V$ denotes the vertex set and $E$ denotes the edge set. A graph is said to be \textit{locally finite} if for any $x\in V$, there are only finite $y\in V$ such that $xy\in E$. A graph is \textit{connected} if any two vertices $x$ and $y$ can be connected via finite edges. For any edge $xy\in E$, we assume the weight $\omega_{xy}$ satisfies $\omega_{xy}=\omega_{yx}>0$. The \textit{degree} of $x\in V$ is defined as $ deg(x)=\sum_{y\sim x}\omega_{xy}$, where we write $y\sim x$ if $xy\in E$. The \textit{distance} $d(x,y)$ of two vertices $x,y\in V$ is defined by the minimal number of edges which connect $x$ and $y$. The measure $\mu:V\to\mathbb{R}^+$ on the graph is a finite positive function on $G$. For any function $u:V\to\mathbb{R}$, the \textit{graph Laplacian} of $u$ is defined by
\begin{equation}\label{Laplacian}
  \Delta u(x)=\frac{1}{\mu(x)}\sum_{y\sim x}\omega_{xy}\left(u(y)-u(x)\right).
\end{equation}

In recent years, the study of partial differential equations on graphs has drawn much attention, see \cite{BSW2022,Huang2012,Lin2017-1,HLY2020,HS2022,Ge2018-1,Grigor'yan2016-2,G2018,Ge2018-2,Grigor'yan2016-1,Grigor'yan2017} and references therein. Specially in \cite{Grigor'yan2017}, Grigor'yan, Lin and Yang studied the following nonlinear Schr\"odinger equation
\begin{equation}\label{eq-f}
  -\Delta u+b(x)u=f(x,u) \quad \mbox{in }V
\end{equation}
on a connected locally finite graph $G$, where the potential $b: V\to \mathbb{R}^+$ satisfies $b(x)\ge b_0$ for some constant $b_0>0$, and one of the following hypotheses holds:
\begin{description}
\item[$(B_1)$] $b(x)\to+\infty$ as $d(x,x_0) \to+\infty$ for some fixed $x_0\in V$;
\item[$(B_2)$]$1/b(x)\in L^1(V)$.
\end{description}
When $f$ satisfies the so-called Ambrosetti-Rabinowitz ((AR) for short) condition:
\begin{description}
\item[$(f_1)$]there exists a constant $\theta>2$ such that for all $x\in V$ and $s>0$,
\begin{eqnarray}\label{AR}
0<\theta F(x,s)=\theta \int_0^s f(x,t)dt\le sf(x,s),
\end{eqnarray}
\end{description}
and some additional assumptions, they applied the mountain pass theorem to show that equation (\ref{eq-f}) admits strictly positive solutions. In \cite{Zhang2018}, Zhang and Zhao established the existence and convergence (as $\lambda\to+\infty$) of ground state solutions for equation (\ref{eq-f}), when $b(x)=\lambda a(x)+1$ and $f(x,u)=\vert u\vert^{p-1}u$, where $a(x)\ge0$ and the potential well $\Omega=\{x\in V: a(x)=0\}$ is a non-empty connected and bounded domain in $V$. One may find more related works in \cite{HSZ2020,LY2022,XZ2021} and their references.

On the other hand, this type of problem has been studied extensively in the Euclidean space, see for example \cite{Ambrosetti1997,AM2006,AMF2005,Bartsch2001,Bartsch1995,CPS2013,Li2006,Rabinowitz1992,Willem1996}.
In recent years, the logarithmic Schr\"{o}dinger equation
\begin{equation}\label{eq-log}
 -\Delta u+b(x)u=u\log u^2\ \ \mbox{in }\mathbb{R}^N
\end{equation}
has attracted much interest. It is closely related to the time-dependent logarithmic Schr\"{o}dinger equation
\[
i\frac{\partial u}{\partial t}-\Delta u+b(x)u-u\log u^2=0\ \ \mbox{in }\mathbb{R}^N\times\mathbb{R}_+,
\]
which has wide applications in quantum mechanics, quantum optics, nuclear physics, transport and diffusion phenomena, open quantum systems, effective quantum gravity, theory of superfluidity and Bose-Einstein condensation, see for example \cite{Carles2018,Cazenave2003,Zloshchastiev2010}.

In fact, there are several challenges when studying the logarithmic Schr\"{o}dinger equation. One of the main challenges is that the associated energy functional of equation (\ref{eq-log}) is not well defined in $H^1(\mathbb{R}^N)$ because there exists $u\in H^1(\mathbb{R}^N)$ such that $\int_{\mathbb{R}^N}u^2\log u^2dx=-\infty$. Different approaches have been developed to overcome this technical difficulty. Cazenave \cite{Cazenave1983} worked in an Orlicz space endowed with a Luxemburg type norm in order to make the functional to be $C^1$. Squassina and Szulkin \cite{Squassina2015} studied the existence of multiple solutions by using non-smooth critical point theory (see also \cite{d'Avenia2014,d'Avenia2015,Ji2016}). Tanaka and Zhang \cite{TZ2017} applied the penalization technique to study multi-bump solutions of equation (\ref{eq-log}). For the idea of penalization, see also Guerrero et al. \cite{Guerrero2010}.
In \cite{Wang2019}, Wang and Zhang proved that the ground state solutions of the power-law scalar field equations $-\Delta u+\lambda u=\vert u\vert^{p-2}u$, as $p\downarrow 2$, converge to the ground state solutions of the logarithmic-law equation $-\Delta u=\lambda u\log u^2$.
By using the direction derivative and constrained minimization method, Shuai \cite{Shuai2019} obtained the existence of positive and sign-changing solutions of equation (\ref{eq-log}) under different types of potentials. For more studies on logarithmic Schr\"odinger equations, one may refer to \cite{AJ2022,AJ2020,C2022,Cazenave1982,ITWZ2021,Shuai2021,ZZ2020} and their references.

The goal of this paper is to investigate the logarithmic Schr\"{o}dinger equation on connected locally finite graphs. The problem (\ref{eq-eq1}) discussed in this paper can be viewed as the discrete model of equation (\ref{eq-log}). By developing variational methods on discrete graphs and establishing new embedding results, we shall obtain the existence of ground state solutions of equation (\ref{eq-eq1}).

 Throughout this paper, we always assume that there exists some constant $\mu_{\min}>0$ such that the measure $\mu(x)\geq\mu_{\min}>0$ for all $x\in V$. For the potential $a(x)$, we assume that the following assumptions hold:
\begin{description}
\item[$(A_1)$] $a: V\to \mathbb{R}$ satisfies $\inf\limits_{x\in V}a(x)\ge a_0$ for some constant $a_0\in(-1, 0)$;
\item[$(A_2)$] for every $M>0$ such that the volume of the set $D_M$ is finite, namely,
\[
Vol(D_M)=\sum\limits_{x\in D_M}\mu(x)<\infty,
\]
where $D_M=\{x\in V:  a(x)\leq M\}$.
\end{description}

Our first result is as follows.
\begin{theorem}\label{th1.1}
Let $G=(V,E)$ be a connected locally finite graph. Assume the potential $a:V\to \mathbb{R}$  satisfies $(A_1)$ and $(A_2)$. Then the problem (\ref{eq-eq1}) admits a ground state solution.
\end{theorem}

We note that the assumption $(A_2)$ is weaker than the coercivity condition $(B_1)$ in \cite{Grigor'yan2017} even when $D_M$ is empty. This assumption was originally introduced by Bartsch and Wang in \cite{Bartsch1995} to study the existence and multiplicity of solutions of nonlinear Schr\"odinger equations in $\mathbb{R}^N$. The potential $a(x)$ in \cite{Bartsch1995} is positive and bounded away from $0$, and later was relaxed to the critical frequency case, namely, $\inf_{x\in V}a(x)=0$, by Byeon and Wang \cite{Byeon2002,Byeon2003} and Sirakov \cite{Sirakov2002}. In \cite{DS2007}, in a similar spirit, Ding and Szulkin dealt with the case when $a(x)$ is sign-changing. Inspired by these studies on nonlinear Schr\"odinger equations in the Euclidean space, we consider problem (\ref{eq-eq1}) when $a(x)$ may change sign. As far as we know, there is no result about nonlinear Schr\"odinger equations on discrete graphs with potentials of sign-changing.

Furthermore, if we generalize $(B_2)$ in \cite{Grigor'yan2017} to the sign-changing potential case, we shall obtain the following result.
\begin{theorem}\label{th1.2}
Let $G=(V,E)$ be a connected locally finite graph. Assume the potential $a:V\to \mathbb{R}$  satisfies $(A_1)$ and
\begin{description}
\item[$(A'_2)$] there exists $M_0>0$ such that $1/a(x)\in L^1(V\setminus D_{M_0})$, where $D_{M_0}=\{x\in V: a(x)\le M_0\}$ and the volume of $D_{M_0}$ is finite.
\end{description}
Then the problem (\ref{eq-eq1}) admits a ground state solution.
\end{theorem}

To prove Theorem \ref{th1.1} and Theorem \ref{th1.2}, the difficulties are mainly twofold.

Firstly, in most previous works about nonlinear Schr\"odinger equations on locally finite graphs, the potentials have positive lower bound. Introducing the corresponding workspace
\[ \mathcal{E}=\left\{u\in \mathcal{D}^{1,2}(V): \int_V a(x)u^2d\mu<+\infty\right\},\]
endowed with $\|u\|_{\mathcal{E}}^2=\int_V\left(\vert\nabla u\vert^2+a(x)u^2\right)d\mu$, we know that $\mathcal{E}$ is continuously embedded into $H^1(V)$. However, due to the assumption that $a(x)$ may change sign, it is not clear if the quantity $\|\cdot\|_{\mathcal{E}}$ is a norm or not and hence the embedding is not available.

Secondly, it is known that the logarithmic Sobolev inequality plays an important role in the study of logarithmic Schr\"odinger equations (see \cite{AJ2020,Shuai2019,Shuai2021,Squassina2015,Wang2019} etc.). A logarithmic Sobolev inequality on discrete graphs is proved under a positive curvature condition, which requires the measure $\mu$ to be finite (see \cite{Lin2017-3} for the details). Unfortunately, this inequality can not be applied here because, in our case, the measure $\mu$ has a uniform positive lower bound.

 Moreover, note that in previous results on connected locally finite graphs, the (AR) condition plays a crucial role in obtaining bounded Palais-Smale sequences for the associated energy functional, whereas because of the logarithmic term, the (AR) condition is not satisfied and it makes our proof more difficult.

Indeed, to deal with the logarithmic term, two Sobolev compact embedding results are established under $(A_1)-(A_2)$ and $(A_1)-(A'_2)$ respectively. In each case, we encounter different obstacles in seeking the solutions of (\ref{eq-eq1}). Specifically, when $a(x)$ satisfies $(A_1)$ and $(A_2)$, the associated energy functional is not well-defined and the previous arguments on locally finite graphs do not work. We borrow the ideas in \cite{Shuai2019,Wang2019} to overcome the difficulty. By restricting $u^2\log u^2\in L^1(V)$, we develop an approach based on the direction derivative and the Nehari manifold method for locally finite graphs to obtain the existence of ground state solutions of (\ref{eq-eq1}). For the other case, i.e., when $a(x)$ satisfies $(A_1)$ and $(A'_2)$, the associated energy functional is $C^1$, and we can apply a variant of mountain pass theorem to obtain the existence of ground state solutions.

The paper is organized as follows. In Sect. \ref{Preliminary}, we present some notations, definitions and lemmas that will be used throughout the paper, and then two Sobolev compact embedding results will be given. In Sect. \ref{Proof of Theorem 1.1}, we develop the Nehari manifold method for locally finite graphs to prove Theorem \ref{th1.1}. In Sect. \ref{Proof of Theorem 1.2}, we apply the mountain pass theorem with the Cerami condition to give the proof of Theorem \ref{th1.2}.  For simplicity, we denote by $C, C_1,  C_2, \cdots$ the positive constants which may be different in different places.

\section{Some preliminary results}\setcounter{equation}{0}\label{Preliminary}

Let us start with some notations. For any function $u: V\to\mathbb{R}$, the integral of $u$ over $V$ is defined by
\begin{equation}\label{L^1}
\int_{V}ud\mu=\sum_{x\in V}\mu(x)u(x).
\end{equation}
The gradient form of two functions $u,v$ on $V$ is defined as
\begin{equation}\label{gradient}
  \Gamma(u,v)(x)=\frac{1}{2\mu(x)}\sum_{y\sim x}\omega_{xy}\left(u(y)-u(x)\right)\left(v(y)-v(x)\right).
\end{equation}
Write $\Gamma(u)=\Gamma(u,u)$, and sometimes we use $\nabla u \nabla v$ to replace $\Gamma(u,v)$. The length of the gradient for $u$ is defined by
\begin{equation}\label{gradient length}
  \vert \nabla u\vert (x)=\sqrt{\Gamma(u)(x)}=\left(\frac{1}{2\mu(x)}\sum_{y\sim x}\omega_{xy}\left(u(y)-u(x)\right)^2\right)^{1/2}.
\end{equation}
Denote by $C_c(V)$ the set of all functions with compact support, and let $H^1(V)$ be the completion of $C_c(V)$ under the norm
\[
\|u\|_{H^1(V)}=\left( \int_V\left(\vert\nabla u\vert^2+u^2\right)d\mu \right)^{1/2}.
\]
Then, $H^1(V)$ is a Hilbert space with the inner product
\[
\langle u,v\rangle=\int_V\left(\Gamma(u,v)+uv\right)d\mu,\ \forall u,\ v\in H^1(V).
\]
We write $\|u\|_p=\left(\int_V\vert u\vert^p d\mu \right)^{1/p}$ for $p\in[1,+\infty)$ and $\|u\|_{L^\infty}=\sup\limits_{x\in V}\vert u(x)\vert$.

We introduce the space
\[
\mathcal{H}=\left\{u\in H^1(V): \int_Va(x)u^2d\mu<+\infty\right\}
\]
with the norm
\[
\|u\|_{\mathcal{H}}=\left(\int_V\left(\vert \nabla u\vert^2+\left(a(x)+1\right)u^2 \right)d\mu\right)^{1/2}
\]
induced by the inner product
\[
\langle u,v\rangle_{\mathcal{H}}=\int_V\left(\Gamma(u,v)+\left(a(x)+1\right)uv\right)d\mu,\ \forall u,\ v\in \mathcal{H}.
\]
Clearly, $\mathcal{H}$ is also a Hilbert space.

Next, we will establish two Sobolev embedding results when $a(x)$ satisfies $(A_1)-(A_2)$ and $(A_1)- (A'_2)$, respectively.

\begin{lemma}\label{lemA2}
Assume $\mu(x)\geq\mu_{\min}>0$ and $a(x)$ satisfies $(A_1)$, $(A_2)$. Then $\mathcal{H}$ is weakly pre-compact and $\mathcal{H}$ is compactly embedded into $L^p(V)$ for all $2\leq p\leq+\infty$.
\end{lemma}
\begin{proof}
At any vertex $x_0\in V$, by $(A_1)$ we have
\[\begin{aligned}
\|u\|_{\mathcal{H}}^2=&\int_V\left(\vert\nabla u\vert^2+\left(1+a(x)\right)u^2\right)d\mu\\
\geq&(1+a_0)\int_Vu^2d\mu\\
=&(1+a_0)\sum\limits_{x\in V}\mu(x)u^2(x)\\
\geq&(1+a_0)\mu_{min}u^2(x_0),
\end{aligned}\]
which implies $\vert u(x_0)\vert \leq\sqrt{\frac{1}{(1+a_0)\mu_{min}}}\|u\|_{\mathcal{H}}$. Thus $\mathcal{H}\hookrightarrow L^\infty(V)$ continuously. Hence, by interpolation we deduce $\mathcal{H}\hookrightarrow L^p(V)$ continuously for all $2\leq p\leq\infty$. Since $\mathcal{H}$ is a Hilbert space and then reflexive, if $\{u_k\}$ is bounded in $\mathcal{H}$, then we have that, up to a subsequence, $u_k\rightharpoonup u$ in $\mathcal{H}$ for some $u\in \mathcal{H}$. In particular, $\{u_k\}\subset\mathcal{H}$ is also bounded in $L^2(V)$ and we have the weak convergence in $L^2(V)$, i.e., for any $\varphi\in L^2(V)$,
\begin{equation}\label{point-wisely}
\lim\limits_{k\to\infty}\int_V(u_k-u)\varphi d\mu=\lim_{k\to\infty}\sum_{x\in V}\mu(x)\left(u_k(x)-u(x)\right)\varphi(x)=0.
\end{equation}
Take any $x_0\in V$ and let
\[
\varphi_0(x)=
\begin{cases}
1, ~ x=x_0,\\
0, ~ x\neq x_0.
\end{cases}
\]
Obviously, $\varphi_0(x)\in L^2(V)$. By substituting $\varphi_0$ into (\ref{point-wisely}), we can get that $\lim\limits_{k\to\infty}u_k(x)=u(x)$ for any fixed $x\in V$.

We now prove $u_k\to u$ in $L^p(V)$ for all $2\leq p\leq\infty$. We first show that $u_k\to u$ in $L^2(V)$. It suffices to prove that $\alpha_k:=\|u_k\|_2^2\to\|u\|_2^2$. We assume that, up to a subsequence, $\alpha_k\to \alpha$ such that $\alpha\geq\|u\|_2^2$. Let $x_0\in V$ be fixed. We claim that for every $\epsilon>0$, there exists $R>0$ such that
\begin{equation}\label{greaterR}
\int_{d(x,x_0)>R}u_k^2(x)d\mu<\epsilon\ \mbox{uniformly in }k.
\end{equation}
Since $\{x\in V:  d(x,x_0)\leq R\}$ is a finite set and $u_k\to u$ for any $x\in V$ as $k\to\infty$, we have that $\lim\limits_{k\to\infty}\int_{d(x,x_0)\leq R}\vert u_k-u\vert^2d\mu=0$. If (\ref{greaterR}) holds, then
\[\begin{aligned}
\int_Vu^2d\mu=&\int_{d(x,x_0)\leq R}u^2d\mu+\int_{d(x,x_0)>R}u^2d\mu\\
\geq&\lim_{k\to\infty}\int_{d(x,x_0)\leq R}u_k^2d\mu\\
\geq&\lim_{k\to\infty}\int_Vu_k^2d\mu-\lim_{k\to\infty}\int_{d(x,x_0)>R}u_k^2d\mu\\
\geq&\alpha-\epsilon.
\end{aligned}\]
Since $\epsilon$ is arbitrary, we get $\|u\|_2^2\ge \alpha$ and hence $u_k\to u$ in $L^2(V)$.

It remains to prove (\ref{greaterR}). For fixed $\epsilon>0$, we choose constants $M>\frac{2}{\epsilon}\sup\limits_{k}\|u_k\|_{\mathcal{H}}^2$, $p\in(1,\infty)$ and
\[
C>\sup_{u\in\mathcal{H}\setminus\{0\}}\frac{\|u\|_{2p}^2}{\|u\|_{\mathcal{H}}^2}.
\]
Denote $V_1:=\{x\in V:  d(x,x_0)>R\}$. By $(A_2)$ it follows that, for $R>0$ large enough,
\[
Vol(V_1\cap D_M):=\sum\limits_{x\in V_1\cap D_M }\mu(x)\leq\left(\frac{\epsilon}{2C\sup\limits_{k}\|u_k\|_{\mathcal{H}}^2}\right)^q,
\]
where $q$ satisfies $\frac{1}{p}+\frac{1}{q}=1$.
Note that
\[
\int_{V_1}u_k^2d\mu=\int_{V_1\setminus D_M}u_k^2d\mu+\int_{V_1\cap D_M}u_k^2d\mu.
\]
For $\int_{V_1\setminus D_M}u_k^2d\mu$, we have
\[
\int_{V_1\setminus D_M}u_k^2d\mu\leq\int_{V_1\setminus D_M}\frac{a(x)}{M}u_k^2d\mu\leq\frac{\|u_k\|_{\mathcal{H}}^2}{M}\leq\frac{\epsilon}{2}.
\]
For $\int_{V_1\cap D_M}u_k^2d\mu$, by H\"{o}lder's inequality
\[\begin{aligned}
\int_{V_1\cap D_M}u_k^2d\mu\leq&\left(\int_{V_1\cap D_M}\vert u_k\vert^{2p}d\mu\right)^{\frac{1}{p}}\left(\int_{V_1\cap D_M}1d\mu\right)^{\frac{1}{q}}\\
\leq&\|u_k\|_{2p}^2\cdot\left(\sum_{x\in V_1\cap D_M}\mu(x)\right)^\frac{1}{q}\\
\leq&C\|u_k\|_{\mathcal{H}}^2\cdot\frac{\epsilon}{2C\sup\limits_{k}\|u_k\|_{\mathcal{H}}^2}\\
\leq&\frac{\epsilon}{2}.
\end{aligned}\]
Hence (\ref{greaterR}) holds.

Since
\[
\|u_k-u\|_{\infty}^2\leq\frac{1}{\mu_{\min}}\int_V\vert u_k-u\vert^2d\mu,
\]
we have for any $2<p<\infty$,
\[
\int_V\vert u_k-u\vert^pd\mu\leq\left(\frac{1}{\mu_{\min}}\right)^{\frac{p-2}{2}}\left(\int_V\vert u_k-u\vert^2d\mu\right)^{\frac{p}{2}}.
\]
This completes the proof.
\end{proof}

\begin{lemma}\label{lemA'2}
Assume $\mu(x)\geq\mu_{\min}>0$ and $a(x)$ satisfies $(A_1)$, $(A'_2)$. Then $\mathcal{H}$ is weakly pre-compact and $\mathcal{H}$ is compactly embedded into $L^p(V)$ for all $1\leq p\leq+\infty$.
\end{lemma}
\begin{proof}
As in the proof of Lemma \ref{lemA2}, we have $\mathcal{H}\hookrightarrow L^p(V)$ continuously for all $2\leq p\leq\infty$. Next, we prove that $\mathcal{H}\hookrightarrow L^p(V)$ continuously for any $1\leq p<2$. Indeed, since $D_{M_0}$ is a bounded set, it follows that $\mathcal{H}\hookrightarrow L^p(D_{M_0})$ continuously for any $1\leq p<2$. By $(A'_2)$, for any $u\in\mathcal{H}$,
\[\begin{aligned}
\int_{V\setminus D_{M_0}}\vert u\vert d\mu=&\int_{V\setminus D_{M_0}}\left(\frac{1}{a(x)}\right)^{\frac{1}{2}}\left(a(x)\right)^{\frac{1}{2}}\vert u\vert d\mu\\
\leq&\left(\int_{V\setminus D_{M_0}}\frac{1}{a(x)}d\mu\right)^{\frac{1}{2}}\left(\int_{V\setminus D_{M_0}}a(x)u^2d\mu\right)^{\frac{1}{2}}\\
\leq&C_1\left(\int_{V\setminus D_{M_0}}\frac{1}{a(x)}d\mu\right)^{\frac{1}{2}}\|u\|_{\mathcal{H}},
\end{aligned}\]
which implies that $\mathcal{H}\hookrightarrow L^p(V\setminus D_{M_0})$ continuously for any $1\leq p<2$.

Let $\{u_k\}$ be a bounded sequence in $\mathcal{H}$. By similar arguments as in Lemma \ref{lemA2}, we have $u_k\to u$ point wisely $x\in V$ as $k\to\infty$. In what follows, we prove $u_k\to u$ in $L^p(V)$ for all $1\leq p\leq\infty$. Since $\{u_k\}$ is bounded in $\mathcal{H}$ and $u\in\mathcal{H}$, there exists some constant $C_2$ such that
\[
\int_{V\setminus D_{M_0}}a(x)\vert u_k-u\vert^2d\mu\leq C_2.
\]
Fix $x_0\in V$. For any $\varepsilon>0$, using $(A'_2)$, there exists $R>0$ such that
\[
\int_{V_1\setminus D_{M_0}}\frac{1}{a(x)}d\mu<\varepsilon^2.
\]
Then, together with the H\"{o}lder's inequality, we deduce
\begin{equation}\label{eqV1bdd}
\begin{aligned}
\int_{V_1\setminus D_{M_0}}\vert u_k-u\vert d\mu=&\int_{V_1\setminus D_{M_0}}\left(\frac{1}{a(x)}\right)^{\frac{1}{2}}\cdot\left(a(x)\right)^{\frac{1}{2}}\vert u_k-u\vert d\mu\\
\leq&\left(\int_{V_1\setminus D_{M_0}}\frac{1}{a(x)d\mu}\right)^{\frac{1}{2}}\left(a(x)\vert u_k-u\vert^2d\mu\right)^{\frac{1}{2}}\\
\leq&\sqrt{C_2}\varepsilon.
\end{aligned}
\end{equation}
Since
\begin{equation*}\label{eqdistbdd}
\lim\limits_{k\to\infty}\left(\int_{V_1\cap D_{M_0}}\vert u_k-u\vert d\mu+\int_{d(x,x_0)\leq R}\vert u_k-u\vert d\mu\right)=0,
\end{equation*}
by (\ref{eqV1bdd}), we get
\[
\liminf\limits_{k\to\infty}\int_V\vert u_k-u\vert d\mu=0,
\]
which implies that $u_k\to u$ in $L^1(V)$. In view of
\[
\|u_k-u\|_{L^\infty(V)}\leq\frac{1}{\mu_{\min}}\int_V\vert u_k-u\vert d\mu,
\]
it follows that
\[
\int_V\vert u_k-u\vert^pd\mu\leq\frac{1}{\mu_{\min}^{p-1}}\left(\int_V\vert u_k-u\vert d\mu\right)^p, \,\, \forall p\in (1, +\infty).
\]
Hence, $u_k\to u$ in $L^p(V)$ for all $1\leq p\leq+\infty$.
\end{proof}

\section{Proof of Theorem \ref{th1.1}}\label{Proof of Theorem 1.1}

In this section, under the assumptions $(A_1)$ and $(A_2)$, we prove that equation (\ref{eq-eq1}) admits a ground state solution by using the Nehari manifold method.

We note that equation (\ref{eq-eq1}) is formally associated with the energy functional $J:\ H^1(V)\to\mathbb{R}\cup\{+\infty\}$ given by
\[
J(u)=\frac{1}{2}\int_V\left(\vert \nabla u\vert^2+\left(a(x)+1\right)u^2 \right)d\mu-\frac{1}{2}\int_Vu^2\log u^2d\mu.
\]
But $J$ may fails to be $C^1$ in $H^1(V)$. In fact, for some $G=(V,E)$ with suitable measure $\mu$, there exists $u\in H^1(V)$ but $\int_Vu^2\log u^2d\mu=-\infty$ (see the 'Appendix' for the details).

 When $a(x)$ satisfies $(A_1)$ and $(A_2)$, we consider the functional $J$ on the set
\[
\mathcal{D}=\left\{u\in \mathcal{H}: \int_Vu^2\vert\log u^2\vert d\mu<\infty\right\},
\]
that is,
 \[
J(u)=\frac{1}{2}\|u\|_{\mathcal{H}}^2-\frac{1}{2}\int_Vu^2\log u^2d\mu, \forall u\in \mathcal{D}.
\]

\begin{definition}\label{def1}
\begin{description}
\item[(1)] Define
\[
J'(u)\cdot v:=\int_V\left(\Gamma(u,v)+a(x)uv\right)d\mu-\int_Vuv\log u^2d\mu, \quad \forall u, v\in\mathcal{D}.
\]
Clearly, $\int_Vuv\log u^2d\mu$ is well-defined for $u,v\in\mathcal{D}$.
\item[(2)] We say that $u\in\mathcal{H}$ is a critical point of $J$ if $u\in\mathcal{D}$ and $J'(u)\cdot v=0$ for all $v\in\mathcal{D}$, and we say that $c\in\mathbb{R}$ is a critical value for $J$ if there exists a critical point $u\in\mathcal{H}$ such that $J(u)=c$.
\end{description}
\end{definition}

Clearly, $u$ is a weak solution to equation (\ref{eq-eq1}) if and only if $u$ is a critical point of $J$. Furthermore,
\begin{proposition}\label{pro-pointsolution}
If $u\in\mathcal{D}$ is a weak solution of (\ref{eq-eq1}), then $u$ is a point-wise solution of equation (\ref{eq-eq1}).
\end{proposition}
\begin{proof}
If $u\in\mathcal{D}$ is a weak solution of (\ref{eq-eq1}), then for any $\varphi\in\mathcal{D}$, there holds
\[
\int_V\left(\Gamma(u,\varphi)+a(x)u\varphi\right)d\mu=\int_Vu\varphi\log u^2d\mu.
\]
Let $C_c(V)$ be the set of all functions on $V$ with compact support. It is easily seen that it is dense in $\mathcal{D}$. By integration by parts we get
\begin{equation}\label{eqpoint-wise}
\int_V\left(-\Delta u+a(x)u\right)\varphi d\mu=\int_Vu\varphi\log u^2d\mu,\ \ \forall \varphi\in C_c(V).
\end{equation}
For any fixed $x_0\in V$, define
\[
\varphi(x)=
\begin{cases}
1, ~ &x=x_0,\\
0, ~ &x\neq x_0.
\end{cases}
\]
Clearly, $\varphi\in C_c(V)$.
Taking $\varphi$ as a test function in (\ref{eqpoint-wise}), we deduce
\[
-\Delta u(x_0)+a(x_0)u(x_0)-u(x_0)\log \left(u(x_0)\right)^2=0.
\]
Since $x_0$ is arbitrary, we conclude that $u$ is a point-wise solution of (\ref{eq-eq1}).
\end{proof}

We define
\[
\mathcal{N}:=\{u\in\mathcal{D}\setminus\{0\}:  J'(u)\cdot u=0\}
\]
and
\[
d=\inf_{u\in\mathcal{N}}J(u).
\]
First we have
\begin{lemma}\label{unique}
For any $u\in \mathcal{D}\setminus\{0\}$, there exists a unique $t_u>0$ such that $t_uu\in\mathcal{N}$. Furthermore, $J(t_uu)>J(tu)$ for all $t\ge0$ but $t\neq t_u$. Specially, if $u\in\mathcal{N}$, then $t_u=1$.
\end{lemma}
\begin{proof}
Fixed $u\in \mathcal{D}\setminus\{0\}$, define $j(t)=J(tu), \forall t\ge0$, i.e.,
\[
j(t)=\frac{t^2}{2}\|u\|_{\mathcal{H}}^2-\frac{t^2}{2}\int_Vu^2\log u^2d\mu -\frac{t^2\log t^2}{2}\int_Vu^2d\mu.
\]
Clearly, $j(0)=0$, and there exist $\delta_2>\delta_1>0$ such that $j(t)>0$ for $t\in (0,\delta_1)$ and $j(t)<0$ for $t\in(\delta_2,+\infty)$.
Hence, $j$ has a maximum at some $t_u\in [\delta_1, \delta_2]$, which implies that $j'(t_u)=0$.
Moreover, we note that
\begin{eqnarray*}
\frac{j'(t)}{t}&=&\frac{J'(tu)\cdot (tu)}{t^2}\\
&=&\|u\|_{\mathcal{H}}^2-\|u\|_2^2-\int_Vu^2\log u^2d\mu-\log t^2\|u\|_2^2,
\end{eqnarray*}
which implies that $\frac{j'(t)}{t}$ is strictly decreasing with respect to $t>0$. Hence $t_u$ is unique.
\end{proof}

Next we prove the following result.
\begin{lemma}
Supposed $(A_1)$ and $(A_2)$ hold. Then $d>0$ is achieved.
\end{lemma}
\begin{proof} By Lemma \ref{unique}, $\mathcal{N}\neq\O$.
Taking a minimizing sequence $\{u_k\}\subset\mathcal{N}$ of $J$ yields
\begin{equation}\label{eqd}
\lim_{k\to+\infty}J(u_k)=\lim_{k\to+\infty}\left[J(u_k)-J'(u_k)\cdot u_k\right]=\lim_{k\to+\infty}\frac{1}{2}\|u_k\|_2^2=d.
\end{equation}
By Lemma \ref{lemA2}, the H\"{o}lder's inequality and Young inequality, for any $\varepsilon\in (0,1)$, there exist $C_{\epsilon}, C'_{\epsilon}, C''_{\epsilon}>0$ such that
\[\begin{aligned}
\int_Vu_k^2\log u_k^2d\mu\leq&\int_V(u_k^2\log u_k^2)^+d\mu\leq C_\varepsilon\int_V\vert u_k\vert^{2+\varepsilon}d\mu\\
\leq&C_\varepsilon\left(\int_V\vert u_k\vert^2d\mu\right)^{\frac{1}{2}}\left(\int_V\vert u_k\vert^{2(1+\varepsilon)}d\mu\right)^{\frac{1}{2}}\\
\leq&C_\varepsilon'\|u_k\|_2\|u_k\|_{\mathcal{H}}^{1+\varepsilon}\\
\leq&\frac{1}{2}\|u_k\|_{\mathcal{H}}^2+C_\varepsilon''\|u_k\|_2^{\frac{2}{1-\varepsilon}}.
\end{aligned}\]
Since $\{u_k\}\subset\mathcal{N}$, we deduce that
\begin{equation}\label{eq1}
\|u_k\|_{\mathcal{H}}^2=\int_Vu_k^2\log u_k^2d\mu+\|u_k\|_2^2\leq\frac{1}{2}\|u_k\|_{\mathcal{H}}^2+C_\varepsilon''\|u_k\|_2^{\frac{2}{1-\varepsilon}}+\|u_k\|_2^2.
\end{equation}
This together with (\ref{eqd}) implies that $\{u_k\}$ is bounded in $\mathcal{H}$. Thus, by Lemma \ref{lemA2}, up to a subsequence, there exists $u_0\in\mathcal{H}$ such that
\[\begin{cases}
u_k\rightharpoonup u_0 ~ \mbox{ weakly in }\mathcal{H},\\
u_k\to u_0 ~ \mbox{ point-wisely in } V,\\
u_k\to u_0 ~ \mbox{ strongly in } L^p(V) \mbox{ for } p\in[2,+\infty].
\end{cases}\]
Then, by the weak-lower semi-continuity of norm and Fatou's lemma, together with the Lebesgue dominated convergence theorem, we get
\[\begin{aligned}
&\int_V\left(\vert\nabla u_0\vert^2+\left(a(x)+1\right)u_0^2\right)d\mu-\int_V\left(u_0^2\log u_0^2\right)^-d\mu\\
&\leq\liminf_{k\to\infty}\left[\int_V\left(\vert\nabla u_k\vert^2+\left(a(x)+1\right)u_k^2\right)dx-\int_V(u_k^2\log u_k^2)^-d\mu\right]\\
&=\liminf_{k\to\infty}\int_V\left[u_k^2+\left(u_k^2\log u_k^2\right)^+\right]d\mu\\
&=\int_Vu_0^2d\mu+\int_V(u_0^2\log u_0^2)^+d\mu.
\end{aligned}\]
It follows that
\begin{equation}\label{equ0}
\int_V\left(\vert \nabla u_0\vert^2+a(x)u_0^2\right)d\mu-\int_Vu_0^2\log u_0^2d\mu\leq0.
\end{equation}
In view of Lemma \ref{unique}, there exists a constant $t_0>0$ such that $t_0u_0\in\mathcal{N}$. By $(\ref{equ0})$, we have
\[\begin{aligned}
0=&J'_\lambda(t_0u_0)\cdot (t_0u_0)\\
=&t_0^2\int_V\left(\vert \nabla u_0\vert^2+a(x)u_0^2\right)d\mu-t_0^2\int_Vu_0^2\log u_0^2d\mu-t_0^2\log t_0^2\|u_0\|_2^2\\
\leq&-t_0^2\log t_0^2\|u_0\|_2^2,
\end{aligned}\]
which implies that $0<t_0\leq1$. Then
\[\begin{aligned}
d\leq& J(t_0u_0)=J(t_0u_0)-\frac{1}{2}J'(t_0u_0)\cdot (t_0u_0)=\frac{1}{2}\|t_0u_0\|_2^2\leq\liminf_{k\to\infty}\frac{t_0^2}{2}\|u_k\|_2^2\\
=&t_0^2\liminf_{k\to\infty}\left[J(u_k)-\frac{1}{2}J'(u_k)\cdot u_k\right]=t_0^2\liminf_{k\to\infty}J(u_k)=t_0^2d\leq d.
\end{aligned}\]
Hence $t_0=1$ and thus $J(u_0)=d$.

We claim that $d>0$. In fact, if $d=0$, then
\[
0=J(u_0)-J'(u_0)\cdot u_0=\frac{1}{2}\|u_0\|_2^2.
\]
By (\ref{eq1}), we have $\|u_0\|_{\mathcal{H}}=0$.

On the other hand, by Lemma \ref{lemA2}, for any $q>2$, there exists $C_{q}>0$ such that
\[
\|u_0\|_{\mathcal{H}}^2=\int_Vu_0^2\log u_0^2d\mu\leq\int_V(u_0^2\log u_0^2)^+d\mu\leq C_q\int_V\vert u_0\vert^qd\mu\leq C\|u_0\|_{\mathcal{H}}^q,
\]
 which implies
\[
\|u_0\|_{\mathcal{H}}\geq \left(\frac{1}{C}\right)^{\frac{1}{q-2}}>0.
\]
Contradiction! Hence the claim holds.
\end{proof}

The following lemma is crucial, which will completes the proof of Theorem \ref{th1.1}.
\begin{lemma}
If $J(u)=c$ for some $u\in\mathcal{N}$, then $u$ is a weak solution of equation (\ref{eq-eq1}).
\end{lemma}
\begin{proof}
By contradiction, we can find a function $\phi\in C_c(V)$ such that
\[
\int_V\left(\nabla u\nabla\phi+a(x)u\phi\right)d\mu-\int_Vu\phi\log u^2d\mu\leq-1,
\]
which implies that, for some $\varepsilon>0$ small enough,
\begin{eqnarray}\label{11-23-1}
J'\left(su+\sigma\phi\right)\cdot\phi\leq-\frac{1}{2},\ \mbox{for all }\vert s-1\vert+\vert\sigma\vert\leq\epsilon.
\end{eqnarray}

In what follows, we estimate $\sup\limits_{s}J\left(su+\varepsilon\eta(s)\phi\right)$, where $\eta$ is a cut-off function such that
\[\eta(s)=
\begin{cases}
1 ~~~ \mbox{if }\vert s-1\vert\leq\frac{1}{2}\varepsilon,\\
0 ~~~ \mbox{if }\vert s-1\vert\geq\varepsilon.
\end{cases}\]
Since $\eta(s)=0$ for $\vert s-1\vert\geq\varepsilon$, it suffices to consider the case $\vert s-1\vert\leq\varepsilon$. Indeed, by (\ref{11-23-1}) we obtain
\[\begin{aligned}
J\left(su+\varepsilon\eta(s)\phi\right)=&J\left(su+\varepsilon\eta(s)\phi\right)-J(su)+J(su)\\
=&J(su)+\int_0^1 J'\left(su+\tau\varepsilon\eta(s)\phi\right)\cdot\left(\varepsilon\eta(s)\phi\right) d\tau\\
=&J(su)+\varepsilon\eta(s)\int_0^1J'\left(su+\tau\varepsilon\eta(s)\phi\right)\cdot \phi d\tau\\
\leq&J(su)-\frac{1}{2}\varepsilon\eta(s),
\end{aligned}\]
which implies that
\[
J\left(u+\varepsilon\eta(1)\phi\right)\leq J(u)-\frac{1}{2}\varepsilon\eta(1)=J(u)-\frac{1}{2}\varepsilon.
\]
In addition, by Lemma \ref{unique} we have $J(su)<J(u)$ for all $s\neq1$. Then
\[
J\left(su+\varepsilon\eta(s)\phi\right)\leq J(su)<J(u)\ \mbox{for all } s\neq1.
\]
Hence, we deduce $J\left(su+\varepsilon\eta(s)\phi\right)<J(u)=c$. Therefore, for $0<\varepsilon<1-\varepsilon$, there exists $s_0\in [\varepsilon, 2-\varepsilon]$ such that
\[
J\left(s_0u+\varepsilon\eta(s)\phi\right)=\sup\limits_{\varepsilon\leq s\leq2-\varepsilon}J\left(su+\varepsilon\eta(s)\phi\right)<c.
\]
Denote $v=su+\varepsilon\eta(s)\phi$ and define $F(s)=J'(v)\cdot v$. It is easily seen that
$F(\varepsilon)>0,\ F(2-\varepsilon)<0$, which implies that there exists $s_1\in[\varepsilon,2-\varepsilon]$ such that $F(s_1)=J'(v)\cdot v\vert_{s=s_1}=0$. Let $\widetilde{u}=s_1u+\varepsilon\eta(s_1)\phi$. Clearly, $\widetilde{u}\in\mathcal{N}$ and $J(\widetilde{u})<c$. This gives a contradiction to the definition of $c$.
\end{proof}

\section{Proof of Theorem \ref{th1.2}}\label{Proof of Theorem 1.2}

This section is devoted to proving Theorem \ref{th1.2}. We recall that $J(u)=\frac{1}{2}\|u\|_{\mathcal{H}}^2-\frac{1}{2}\int_Vu^2\log u^2d\mu$. Note that, for any $0<\varepsilon<1$, there exists $C_{\epsilon}>0$ such that
\begin{equation}\label{eqLogleq}
\vert u^2\log u^2\vert \leq C_\epsilon\left(\vert u\vert^{2-\epsilon}+\vert u\vert^{2+\epsilon}\right).
\end{equation}
When $a(x)$ satisfies $(A_1)$ and $(A'_2)$, from Lemma \ref{lemA'2}, the embedding $\mathcal{H}\hookrightarrow L^p(V)$ is compact for $p\in[1,+\infty]$.
Then, in view of \cite[Lemma 2.16]{Willem1996}, we can easily conclude the following result.
\begin{proposition}\label{proposition}
Assume that $a(x)$ satisfies $(A_1)$ and $(A'_2)$. Then the functional $J\in C^1(\mathcal{H},\mathbb{R})$. Moreover, if $u\in\mathcal{H}$ is a critical point of $J$, then it is a weak solution of equation (\ref{eq-eq1}), i.e.,
\[
\int_V\left(\Gamma(u,v)+a(x)uv \right)\,d\mu=\int_Vuv\log u^2d\mu,\ \ \forall v\in\mathcal{H}.
\]
\end{proposition}

Similar as Proposition \ref{pro-pointsolution}, we have following result.
\begin{proposition}
If $u\in\mathcal{H}$ is a weak solution of (\ref{eq-eq1}), then $u$ is a point-wise solution of equation (\ref{eq-eq1}).
\end{proposition}

 We recall that, for $c\in\mathbb{R}$, $J$ satisfies the Cerami condition at level $c$, if for any sequence $\{u_k\}\subset \mathcal{H}$ with
\begin{equation}\label{eqCerami}
J(u_k)\to c,\ \
(1+\|u_k\|_{\mathcal{H}})\|J'(u_k)\|_{\mathcal{H}'}\to0
\end{equation}
there is a subsequence $\{u_k\}$ such that $\{u_k\}$ converges strongly in $\mathcal{H}$. Any sequence $\{u_k\}\subset \mathcal{H}$ for which (\ref{eqCerami}) holds true is called a Cerami sequence of $J$ at level $c$.

To prove Theorem \ref{th1.2}, we shall apply the following version of mountain pass theorem, which provides the existence of a Cerami sequence at the mountain pass level.
\begin{lemma}\label{leMP}\cite{Schechter1991}
Let $X$ be a real Banach space, $\Phi\in C^1(X,\mathbb{R})$ satisfies
\begin{eqnarray*}
\max\{\Phi(0), \Phi(e)\le \alpha<\beta\le \inf\limits_{\|u\|_X=\rho}\Phi(u)
\end{eqnarray*}
for some $\alpha, \beta, \rho>0$ and $e\in X$ with $\|e\|_X>\rho$. Set
\[
c=\inf_{\gamma\in\Gamma}\max_{0\leq\tau\leq1}\Phi\left(\gamma(\tau)\right),
\]
where
\[
\Gamma=\left\{\gamma\in C\left([0,1],X\right):  \gamma(0)=0 \, \mbox{ and } \, \gamma(1)=e\right\}.
\]
Then there exists a Cerami sequence $\{u_k\}\subset \mathcal{H}$ of $\Phi$ at level $c$.
\end{lemma}

Firstly, we show $J$ admits a mountain pass geometry.
\begin{lemma}\label{lemMpg1}
\begin{description}
\item[(i)]There exists positive constants $\rho$ and $\delta$ such that $J(u)\geq\delta$ for all $u\in \mathcal{H}$ with $\|u\|_{\mathcal{H}}=\rho$.
\item[(ii)]There exists $\varphi\in\mathcal{H}\setminus \{0\}$ such that $J(t\varphi)\to-\infty$ as $t\to+\infty$.
\end{description}
\end{lemma}
\begin{proof}
For (i), by the definition of $J$ and Lemma \ref{lemA'2}, for any $q>2$, there exist constants $C_q, C>0$ such that
\[\begin{aligned}
J(u)=&\frac{1}{2}\|u\|_{\mathcal{H}}^2-\frac{1}{2}\int_Vu^2\log u^2d\mu\\
\geq&\frac{1}{2}\|u\|_{\mathcal{H}}^2-\frac{1}{2}\int_V(u^2\log u^2)^+d\mu\\
\geq&\frac{1}{2}\|u\|_{\mathcal{H}}^2-\frac{C_q}{2}\|u\|_q^q\\
\geq&\frac{1}{2}\|u\|_{\mathcal{H}}^2-\frac{C}{2}\|u\|_{\mathcal{H}}^q.
\end{aligned}\]
Taking $\rho>0$ small enough, there exits a constant $\delta>0$ such that
\[
J(u)\geq\delta
\]
for all $u\in \mathcal{H}$ with $\|u\|_{\mathcal{H}}=\rho$.

For (ii), for any $\varphi\in \mathcal{H}\setminus \{0\}$, we have
\[
J(t\varphi)=\frac{t^2}{2}\|\varphi\|_{\mathcal{H}}^2-\frac{t^2}{2}\int_V\varphi^2\log \varphi^2d\mu-\frac{t^2}{2}\log t^2\|\varphi\|_2^2.
\]
This shows that $J(t\varphi)\to-\infty$ as $t\to+\infty$.
\end{proof}

Next, we prove that $J$ satisfies the Cerami condition.
\begin{lemma}\label{lemPS}
The functional $J$ satisfies the Cerami condition at any level $c>0$.
\end{lemma}
\begin{proof}
Assume that $\{u_k\}\subset \mathcal{H}$ is a Cerami sequence of $J$. We claim that $\{u_k\}$ is bounded in $\mathcal{H}$.
If not, we may assume that $\{u_k\}$ is unbounded in $\mathcal{H}$. Set $w_k=\frac{u_k}{\|u_k\|_{\mathcal{H}}}$. Then, passing to a subsequence if necessary, we may assume that there exists $w\in\mathcal{H}$ such that
\[\begin{cases}
w_k\rightharpoonup w &\mbox{ in } \mathcal{H},\\
w_k\to w &\mbox{ point-wisely in } V,\\
w_k\to w &\mbox{ in } L^p(V), p\in[1,+\infty].
\end{cases}\]
We will distinguish the following two cases:
\begin{description}
\item[Case 1: $w=0$.]

Let $t_k\in[0,1]$ such that
\[
J(t_ku_k)=\max_{t\in[0,1]}J(tu_k).
\]
By the unboundedness of $\{u_k\}$, for any given $\tau>0$, there exists $N>0$ such that
\[
\frac{\tau}{\|u_k\|_{\mathcal{H}}}\in(0,1),\ k\geq N.
\]
Denote $\overline{w}_k=(4\tau)^{\frac{1}{2}}w_k$. By (\ref{eqLogleq}) and the Lebesgue dominate convergence theorem it follows that
\[\begin{aligned}
&\lim_{k\to\infty}\int_V(\overline{w}_k^2\log \overline{w}_k^2)d\mu\\
=&\lim_{k\to\infty}\left(4\tau\int_Vw_k^2\log w_k^2d\mu+4\tau\log(4\tau)\int_Vw_k^2d\mu\right)\\
=&4\tau\left(\int_V w^2\log w^2d\mu+\log(4\tau)\int_Vw^2d\mu\right)\\
=&0,
\end{aligned}\]
which implies that, for $k$ large enough,
\[\begin{aligned}
J(t_ku_k)\geq&J\left(\frac{(4\tau)^{\frac{1}{2}}}{\|u\|_{\mathcal{H}}}u_k\right)=J(\overline{w}_k)\\
=&\frac{1}{2}\|\overline{w}_k\|_{\mathcal{H}}^2-\frac{1}{2}\int_V\overline{w}_k^2\log\overline{w}_k^2d\mu\\
\ge&\tau.
\end{aligned}\]
Since $\tau>0$ is arbitrary, we deduce
\begin{eqnarray}\label{11-24-1}
\lim_{k\to\infty}J(t_ku_k)=+\infty.
\end{eqnarray}
However, by $J(0)=0$ we have $t_k\in(0,1)$. Then $\frac{d}{dt}J(tu_k)\vert_{t=t_k}=0$. Hence we obtain
\[\begin{aligned}
J(t_ku_k)=&J(t_ku_k)-\frac{1}{2} J'(t_ku_k)\cdot (t_ku_k)\\
=&\frac{1}{2}\int_V\vert t_ku_k\vert^2d\mu\\
\leq&\frac{1}{2}\int_Vu_k^2d\mu\\
=&J(u_k)-\frac{1}{2} J'(u_k)\cdot u_k\\
\leq&C,
\end{aligned}\]
which is contrary to (\ref{11-24-1}).
\item[Case 2: $w\neq0$.]
Set $V'=\{x\in V; w\neq0\}$. Then $\vert u_k(x)\vert\to+\infty$ point-wisely in $V'$. Since $\{\|u_k\|_{\mathcal{H}}\}$ is unbounded and $J(u_k)\leq c$, we have $\frac{J(u_k)}{\|u_k\|_{\mathcal{H}}^2}\to0$, i.e.,
\[
\frac{1}{2}-\frac{1}{2}\int_{V'}\frac{u_k^2\log u_k^2}{\|u_k\|_{\mathcal{H}}^2}d\mu-\frac{1}{2}\int_{V\setminus V'}\frac{u_k^2\log u_k^2}{\|u_k\|_{\mathcal{H}}^2}d\mu=o_k(1),
\]
which is equivalent to the following
\begin{equation}\label{eqwneq0}
\begin{aligned}
&C-\int_{\{x\in V\setminus V';\vert u_k(x)\vert\leq1\}}\frac{u_k^2\log u_k^2}{\|u_k\|_{\mathcal{H}}^2}d\mu\\
=&\int_{\{x\in V\setminus V';\vert u_k(x)\vert>1\}}\frac{u_k^2\log u_k^2}{\|u_k\|_{\mathcal{H}}^2}d\mu+\int_{V'}\frac{u_k^2\log u_k^2}{\|u_k\|_{\mathcal{H}}^2}d\mu.
\end{aligned}
\end{equation}
On the left-side of the (\ref{eqwneq0}), from the definition of $w_k$, we know $\vert u_k(x)\vert<+\infty$ point-wisely in $V\setminus V'$. By (\ref{eqLogleq}) and Lemma \ref{lemA'2} we deduce
\[\begin{aligned}
0\leq&\lim_{k\to\infty}\int_{\{x\in V\setminus V';\vert u_k(x)\vert \leq1\}}\frac{-u_k^2\log u_k^2}{\|u_k\|_{\mathcal{H}}^2}d\mu\\
\leq&\lim_{k\to\infty}\int_{\{x\in V\setminus V';\vert u_k(x)\vert\leq1\}}\frac{2C_\varepsilon\vert u_k\vert^{2-\varepsilon}}{\|u_k\|_{\mathcal{H}}^2}d\mu\\
\leq&\lim_{k\to\infty}\int_{\{x\in V\setminus V';\vert u_k(x)\vert\leq1\}}\frac{2C_\varepsilon\vert u_k\vert}{\|u_k\|_{\mathcal{H}}^2}d\mu\\
\leq&\lim_{k\to\infty}\frac{2C_\varepsilon C}{\|u_k\|_{\mathcal{H}}}\\
=&0.
\end{aligned}\]
On the right-side of the (\ref{eqwneq0}):
\[\begin{aligned}
&\lim_{k\to\infty}\left[\int_{\{x\in V\setminus V';\vert u_k(x)\vert>1\}}\frac{u_k^2\log u_k^2}{\|u_k\|_{\mathcal{H}}^2}d\mu+\int_{V'}\frac{u_k^2\log u_k^2}{\|u_k\|_{\mathcal{H}}^2}d\mu\right]\\
\ge&\lim_{k\to\infty}\int_{V'}w_k^2\log u_k^2d\mu=+\infty,
\end{aligned}\]
which provides a contradiction.
\end{description}
Thus $\{u_k\}$ is bounded in $\mathcal{H}$. Hence, by Lemma \ref{lemA'2}, there exists $u\in\mathcal{H}$ such that, up to a subsequence,
\[\begin{cases}
u_k\rightharpoonup u &\mbox{ in } \mathcal{H},\\
u_k\to u &\mbox{ point-wisely in } V,\\
u_k\to u &\mbox{ in } L^p(V), p\in[1,+\infty].
\end{cases}\]
Since
\begin{equation}\label{eqconver}
\begin{aligned}
&o_k(1)\|u_k-u\|_{\mathcal{H}}\\
=&J'(u_k)\cdot(u_k-u)-J'(u)\cdot(u_k-u)\\
=&\left[\int_V\left(\nabla u_k\nabla(u_k-u)+a(x)u_k(u_k-u)\right)d\mu-\int_Vu_k(u_k-u)\log u_k^2d\mu\right]\\
&-\left[\int_V\left(\nabla u\nabla(u_k-u)+a(x)u(u_k-u)\right)d\mu-\int_Vu(u_k-u)\log u_k^2d\mu\right]\\
=&\|u_k-u\|_{\mathcal{H}}^2-\int_V(u_k-u)u_k\log u_k^2d\mu+\int_V(u_k-u)u\log u^2d\mu-\|u_k-u\|_2^2.\\
\end{aligned}
\end{equation}
By using H\"{o}lder's inequality and Lemma \ref{lemA'2}, we obtain
\[\begin{aligned}
&\lim_{k\to\infty}\int_V(u_k-u)u_k\log u_k^2d\mu\\
\leq&\lim_{k\to\infty}\left(\int_V\vert u_k-u\vert^2d\mu\right)^{\frac{1}{2}}\left(\int_Vu_k^2(\log u_k^2)^2d\mu\right)^{\frac{1}{2}}\\
\leq&\lim_{k\to\infty}\left(\int_V\vert u_k-u\vert^2d\mu\right)^{\frac{1}{2}}\cdot \left(C_\varepsilon\int_V(\vert u_k\vert^{2-2\varepsilon}+\vert u_k\vert^{2+2\varepsilon})d\mu\right)^\frac{1}{2}\\
=&0.
\end{aligned}\]
Similarly, we get
\[
\lim_{k\to\infty}\int_V(u_k-u)u\log u^2d\mu=0.
\]
Letting $k\to\infty$ in (\ref{eqconver}), we have
\[
\lim_{k\to\infty}\|u_k-u\|_{\mathcal{H}}^2=0.
\]
The proof is complete.
\end{proof}

\begin{proof}[Completion of the proof of Theorem \ref{th1.1}.]
By Lemma \ref{lemMpg1}, taking $e=t_1\varphi$ for some $t_1>0$ large enough, we can apply Lemma \ref{leMP} to conclude that there exists
a Cerami sequence of $J$ at level $c=\inf_{\gamma\in\Gamma}\max_{u\in\gamma}J(u)$, where $\Gamma$ is defined as in Lemma \ref{leMP}.
Using Lemma \ref{lemPS} it follows that $c$ is a critical value of $J$, i.e., there exists some $u\in\mathcal{H}\setminus \{0\}$ such that $J(u)=c$ and $J'(u)=0$.

To get ground state solution, we denote by $K$ the critical set of $J$. Set
\[
m=\inf\left\{J(u):  u\in K\setminus\{0\}\right\}.
\]
For any $u\in K$, we have
\[
J(u)=J(u)-\frac{1}{2}J'(u)\cdot u=\frac{1}{2}\|u\|_2^2\geq0,
\]
which implies $m\geq0$. On the other hand, using the Fatou's lemma, we have
\[\begin{aligned}
J(u)=&J(u)-\frac{1}{2}J'(u)\cdot u\\
=&\frac{1}{2}\|u\|_2^2\\
\leq&\frac{1}{2}\liminf_{k\to\infty}\|u_k\|_2^2\\
=&\liminf_{k\to\infty}\left[J(u_k)-\frac{1}{2}J'(u_k)\cdot u_k\right]\\
=&c.
\end{aligned}\]
Therefore, $0\leq m\leq c$.

Next, let $\{v_k\}$  be a sequence of non-trivial critical points of $J$ satisfying
\[
J(v_k)\to m.
\]
By Lemma \ref{lemPS} we can see that $v_k$ converges to some $v\neq0$.  Using the Fatou's lemma again, we obtain
\[
m\leq J(v)\leq\liminf_{k\to\infty}J(v_k)=m.
\]
Hence $J(v)=m$ and $J'(v)=0$.
\end{proof}

\section*{Appendix}
In this appendix, we present two examples to show that there exists $u\in H^1(V)$ but $\int_{V}u^2\log u^2dx=-\infty$.

\begin{example}
We consider a connected locally finite graph $G=(V,E)$ such that $V:=\mathbb{N} \cup \{0\}$, where $\mathbb{N}$ denotes the set of natural numbers.
Fixed $x_0=0\in V$. Let
\[
u(x)=
\begin{cases}
\left(\vert x\vert\log \vert x\vert\right)^{-1}, ~ &\vert x\vert\geq3,\\
0, ~ &\vert x\vert\leq2,
\end{cases}
\]
where $\vert x\vert:=d(x,x_0)$, the measure
\[
\mu(x)=
\begin{cases}
x, ~ &x\geq1,\\
1, ~ &x=0.
\end{cases}
\]
For simplicity, we assume that $\omega_{xy}=1$.

Let's recall the following fact
\[
\sum_{n=2}^{\infty}\frac{1}{n(\log n)^p}
\begin{cases}
<\infty,~~~&\mbox{ if }p>1,\\
=\infty,~~~&\mbox{ if }0\leq p\leq1.
\end{cases}\]

First of all, we prove that $u\in H^1(V)$. By the definition of $u$ and $\mu$, we have
\[
\int_V u^2d\mu=\sum_{x\in \mathbb{N}}\mu(x)u^2(x)=\sum_{x\geq3}\frac{x}{x^2(\log x)^2}=\sum_{x\geq3}\frac{1}{x(\log x)^2}<\infty,
\]
and
\[\begin{aligned}
&\int_V\vert\nabla u\vert^2d\mu\\
=&\frac{1}{2}\sum_{x\in \mathbb{N}}\sum_{y\sim x}\left(u(y)-u(x)\right)^2\\
=&\frac{1}{2}\left(u(3)-u(2)\right)^2+\frac{1}{2}\left(u(2)-u(3)\right)^2+\frac{1}{2}\left(u(4)-u(3)\right)^2\\
&+\frac{1}{2}\sum_{x\geq4}\sum_{y\sim x}\left(u(y)-u(x)\right)^2\\
=&\frac{1}{3^2(\log3)^2}+\sum_{x\geq3}\left(\frac{1}{(x+1)\log (x+1)}-\frac{1}{x\log x}\right)^2\\
\leq&\frac{1}{3^2(\log3)^2}+\sum_{x\geq3}\left(\frac{1}{(x+1)^2\left(\log (x+1)\right)^2}+\frac{1}{x^2(\log x)^2}\right)\\
\leq&\frac{1}{3^2(\log3)^2}+\sum_{x\geq3}\frac{2}{x^2(\log x)^2}\\
<&\infty.
\end{aligned}\]

Next, we prove that $\int_Vu^2\log u^2d\mu=-\infty$. Note that
\[
u^2(x)\log u^2(x)=-\left(\frac{2}{x^2\log x}+\frac{2\log(\log x)}{x^2(\log x)^2}\right),
\]
we have
\[\begin{aligned}
\int_Vu^2\log u^2d\mu=&\sum_{x\in \mathbb{N}}\mu(x)u^2(x)\log u^2(x)\\
=&\sum_{x\geq3}x\cdot u^2(x)\log u^2(x)\\
=&-\left(\sum_{x\geq3}\frac{2}{x\log x}+\sum_{x\geq3}\frac{2\log(\log x)}{x(\log x)^2}\right)\\
:=&-(I+II).
\end{aligned}\]
Since $II>0$, it suffices to prove the following fact
\[
\sum_{x\geq3}\frac{2}{x\log x}=\infty.
\]
But this is obvious, and thus we completes the proof.
\end{example}

\begin{example}
Assume that $G=(V,E)$ is a connected locally finite graph, and there exist some constants $\mu_{\min},\ \mu_{\max}>0$ such that $0<\mu_{\min}\leq\mu(x)\leq\mu_{\max}$ for all $x\in V$.

Fixed $x_0\in V$. Let
\[
u(x)=
\begin{cases}
\left(\vert x\vert^{\frac{1}{2}}\log \vert x\vert\right)^{-1}, ~ &\vert x\vert\geq3,\\
0, ~ &\vert x\vert\leq2,
\end{cases}
\]
where $\vert x\vert:=d(x,x_0)$. Then $u(x)\in H^1(V)$ and $\int_Vu^2\log u^2d\mu=-\infty$. In fact, from the definition of $u$ and $\mu$, we have
\[
\int_V u^2d\mu=\sum_{x\in V}\mu(x)u^2(x)=\sum_{\vert x\vert\geq3}\frac{\mu(x)}{\vert x\vert\left(\log\vert x\vert\right)^2}\leq\mu_{\max}\sum_{\vert x\vert\geq3}\frac{1}{\vert x\vert\left(\log\vert x\vert\right)^2}<\infty,
\]
and
\[\begin{aligned}
&  \int_V\vert\nabla u\vert^2d\mu \\
=&\frac{1}{2}\sum_{x\in V}\sum_{y\sim x}\omega_{xy}\left(u(y)-u(x)\right)^2\\
=&\frac{1}{2}\sum_{\vert x\vert=2}\sum_{y\sim x \atop \vert y\vert=3}\frac{\omega_{xy}}{3(\log3)^2}+\frac{1}{2}\sum_{\vert x\vert=3}\sum_{y\sim x \atop \vert y\vert=2}\frac{\omega_{xy}}{3(\log3)^2}\\
&+\frac{1}{2}\sum_{\vert x\vert=3}\sum_{y\sim x \atop \vert y\vert=4}\omega_{xy}\left(\frac{1}{\sqrt{4}\log4}-\frac{1}{\sqrt{3}\log3}\right)^2+\frac{1}{2}\sum_{\vert x\vert\geq4}\sum_{y\sim x}\omega_{xy}\left(u(y)-u(x)\right)^2\\
=&\sum_{\vert x\vert=3}\sum_{y\sim x \atop \vert y\vert=2}\frac{\omega_{xy}}{3(\log3)^2}+\sum_{\vert x\vert\geq3}\sum_{y\sim x \atop \vert y\vert=\vert x\vert+1}\omega_{xy}\left(\frac{1}{\left(\vert x\vert+1\right)^{\frac{1}{2}}\log\left(\vert x\vert+1\right)}-\frac{1}{\vert x\vert^{\frac{1}{2}}\log\vert x\vert}\right)^2\\
\leq&\sum_{\vert x\vert=3}\sum_{y\sim x \atop \vert y\vert=2}\frac{\omega_{xy}}{3(\log3)^2}+\omega_{\max}\sum_{\vert x\vert\geq3}\sum_{y\sim x \atop \vert y\vert=\vert x\vert+1}\left(\frac{1}{\left(\vert x\vert+1\right)\left(\log\left(\vert x\vert+1\right)\right)^2}+\frac{1}{\vert x\vert\left(\log\vert x\vert\right)^2}\right)\\
\leq&\sum_{\vert x\vert=3}\sum_{y\sim x \atop \vert y\vert=2}\frac{\omega_{xy}}{3(\log3)^2}+\omega_{\max}\sum_{\vert x\vert\geq3}\sum_{y\sim x \atop \vert y\vert=\vert x\vert+1}\frac{2}{\vert x\vert\left(\log\vert x\vert\right)^2}\\
<&\infty.
\end{aligned}\]
In what follows, we prove that $\int_Vu^2\log u^2d\mu=-\infty$. By direct calculations, we have
\[\begin{aligned}
\int_Vu^2\log u^2d\mu=&\sum_{x\in V}\mu(x)u^2(x)\log u^2(x)\\
=&\sum_{\vert x\vert\geq3}\mu(x)u^2(x)\log u^2(x)\\
=&-\left(\sum_{\vert x\vert\geq3}\frac{\mu(x)}{\vert x\vert\log\vert x\vert}+\sum_{\vert x\vert\geq3}\frac{2\mu(x)\log\left(\log\vert x\vert\right)}{\vert x\vert\left(\log\vert x\vert\right)^2}\right)\\
=&-\infty.
\end{aligned}\]
\end{example}

\bmhead{Acknowledgments}
The research of Xiaojun Chang is supported by National Natural Science Foundation of China (No.11971095), while Duokui Yan is supported by National Natural Science Foundation of China (No.11871086). This work was done when Xiaojun Chang visited the Laboratoire de Math\'ematiques, Universit\'e de Bourgogne Franche-Comt\'e during the period from 2021 to 2022 under the support of China Scholarship Council (202006625034), and he would like to thank the Laboratoire for their support and kind hospitality.

\section*{Declarations}

\begin{itemize}
\item Conflict of interest: The authors declare that they have no competing interests.
\end{itemize}

\end{document}